\documentclass[12pt]{article}

\usepackage{amssymb}
\usepackage{amsfonts}
\usepackage{amsthm}
\usepackage{amsmath}
\usepackage{float}
\usepackage{bigints}
\usepackage{fullpage}
\usepackage{amsopn}
\usepackage[dvips]{graphicx}   
\usepackage{color,epsfig}      

\newtheorem{Lemma}{Lemma}

\newtheorem{Theorem}{Theorem}
\newtheorem{Conjecture}{Conjecture}
\newtheorem{Definition}{Definition}

\newtheorem{Remark}{Remark}

\numberwithin{Subcase}{Case}

\DeclareMathOperator{\vol}{vol}

\newcommand{\Eu}{\mathbb{E}}

\newcommand{\BB}{\mathbf{B}}

\newcommand{\M}{\mathbb{M}}
\newcommand{\MM}{\mathbf{M}}

\renewcommand{\P}{\mathcal{P}}
\newcommand{\PP}{\mathbf{P}}
\newcommand{\z}{\mathbf{z}}
\newcommand{\y}{\mathbf{y}}
\newcommand{\x}{\mathbf{x}}
\newcommand{\q}{\mathbf{q}}
\newcommand{\p}{\mathbf{p}}

\renewcommand{\o}{\mathbf{o}}

\DeclareMathOperator{\conv}{conv}

\DeclareMathOperator{\area}{area}
\DeclareMathOperator{\bd}{bd}

\DeclareMathOperator{\inter}{int}

\DeclareMathOperator{\cl}{cl}

\newcommand{\C}{\mathcal{C}}
\renewcommand{\c}{\mathbf{c}}
\renewcommand{\a}{\mathbf{a}}
\renewcommand{\b}{\mathbf{b}}

\title{Remarks on soft ball packings in dimensions 2 and 3
\footnote{\hspace*{.35cm} 2010 Mathematics Subject Classification: 52C15, 52C17, 52C05, 52C07, 52A40.}}
\author{K\'{a}roly Bezdek\thanks{Partially supported by a Natural Sciences and 
Engineering Research Council of Canada Discovery Grant.} and Zsolt L\'angi\thanks{Partially supported by the National Research, Development and Innovation Office, NKFI, K-147544 and BME Water Sciences \& Disaster Prevention TKP2020 Institution Excellence Subprogram, grant no. TKP2020 BME-IKA-VIZ.}
}

\date{}

\begin{document}

\maketitle

\begin{abstract}
We study translative arrangements of centrally symmetric convex domains in the plane (resp., of congruent balls in the Euclidean $3$-space) that neither pack nor cover. We define their soft density depending on a soft parameter and prove that the largest soft density for soft translative packings of a centrally symmetric convex domain with $3$-fold rotational symmetry and given soft parameter is obtained for a proper soft lattice packing. Furthermore, we show that among the soft lattice packings of congruent soft balls with given soft parameter the soft density is locally maximal for the corresponding face centered cubic (FCC) lattice.

\end{abstract}

\section{Introduction}

Motivated by biological findings, lattice arrangements of congruent balls in Euclidean $3$-space with pairwise overlaps limited by a properly introduced parameter (i.e., lattice soft sphere packings), have been studied from the point of view of largest soft densities. For details see \cite{Ed-IgH} and the references mentioned there. A similar, but different concept of soft density has been investigated for not necessarily lattice-type soft packings in \cite{Be-La-2015}. In this note we continue the line of research started in \cite{Be-La-2015} and study the highest soft densities of soft ball packings in Euclidean $d$-space for $d=2,3$. In fact, our results are based on the following somewhat more general concept of soft packings and their soft densities.

\begin{Definition}\label{defn:main}
Let the Euclidean norm of the $d$-dimensional Euclidean space $\Eu^d$ be denoted by $\|\cdot\|$ and let $\BB^d:=\{\x\in \Eu^d | \|\x\|\leq 1\}$ denote the (closed) $d$-dimensional unit ball centered at the origin $\o$ in  $\Eu^d$, where $d>1$. Next,
let $\MM$ be an arbitrary $\o$-symmetric convex body (i.e., let $\MM$ be a compact, convex set with non-empty interior, and $\MM=-\MM$) in $\Eu^d$, $d>1$ and let $\M$ be the $d$-dimensional Minkowski space (i.e., $d$-dimensional normed space) with $\MM$ as its unit ball, and norm $\| \x \|_{\MM}:=\min\{\mu>0 | \x\in\mu\MM\}$ for any $\x\in\Eu^d$.
Then let $\P:=\{\c_i+\MM\ |\ i=1,2,\dots \ {\rm with}\ \|\c_j-\c_k\|_{\M} \ge 2\ {\rm for}\ {\rm all}\ 1\le j< k\}$ be an arbitrary infinite packing of unit balls in $\M$. Moreover, for any $d\ge 2$  and $\lambda\ge 0$ let $\mathbf{P}_{\lambda}:=\bigcup_{i=1}^{+\infty}(\c_i+(1+\lambda)\MM)$ denote the outer parallel domain of $\mathbf{P}:=\bigcup_{i=1}^{+\infty}(\c_i+\MM)$ having outer radius $\lambda$ in $\M$. Furthermore, let
$$\bar{\delta}_d(\mathbf{P}_{\lambda}):=\limsup_{R\to +\infty}\frac{\vol_d(\mathbf{P}_{\lambda}\cap R\BB^d)}{\vol_d(R\BB^d) }$$
be the (upper) density of the outer parallel domain $\mathbf{P}_{\lambda}$ assigned to the unit ball packing $\P$ in $\M$, where $\vol_d(\cdot)$ denotes the $d$-dimensional volume (i.e., Lebesgue measure) of the corresponding set in $\Eu^d$. Finally, let
$$\bar{\delta}_{\M}(\lambda):=\sup_{\P}\bar{\delta}_d(\mathbf{P}_{\lambda})$$
be the largest density of the outer parallel domains of unit ball packings having outer radius $\lambda$ in $\M$. Putting it somewhat differently, one could say that the family
$\{\c_i+(1+\lambda)\MM\ |\ i=1,2,\dots\}$ of closed balls of radii $1+\lambda$ is a packing of soft balls with soft parameter $\lambda$  in $\M$, if $\P:=\{\c_i+\MM\ |\ i=1,2,\dots\}$
is a unit ball packing  of $\M$ in the usual sense. In particular, $\bar{\delta}_d(\mathbf{P}_{\lambda})$ is called the (upper) soft density of the soft ball packing $\{\c_i+(1+\lambda)\MM\ |\ i=1,2,\dots\}$  in $\M$ with
$\bar{\delta}_{\MM}(\lambda)$ standing for the largest soft density of packings of soft balls of radii $1+\lambda$ having soft parameter $\lambda$ in $\M$.
\end{Definition}

\begin{Remark}
We define $\bar{\delta}_{\M}^{\mathrm{lattice}}(\lambda)$ as the supremum $\sup_{\P}\bar{\delta}_d(\mathbf{P}_{\lambda})$ of all outer parallel domains $\mathbf{P}_{\lambda}$ of all lattice packings $\P$ of unit balls in $\M$, having outer radius $\lambda$. 
\end{Remark}

\begin{Remark}
For any normed space $\M$, the functions $\bar{\delta}_{\M}(\lambda)$ and $\bar{\delta}_{\M}^{\mathrm{lattice}}(\lambda)$ are increasing functions of $\lambda$, the value of $\bar{\delta}_{\M}(0)$ is the packing density of $\MM$, and if $\gamma(\MM)$ denotes the simultaneous packing and covering constant of $\MM$ (\cite{Zo2008}), then $\gamma(\MM)-1$ is the
smallest value of $\lambda$ with $\bar{\delta}_{\M}(\lambda) = 1$.
\end{Remark}

By \cite{FTL} (see also \cite{Rogers}, \cite{Rogers2}) and \cite{Zo2008}, for any $\o$-symmetric planar convex body $\MM$, we have that both $\bar{\delta}_{\M}(0)$ and $\gamma(\MM)$ are attained for lattice packings of $\MM$. The following theorem extends these results to $\o$-symmetric planar convex bodies with $3$-fold rotational symmetry as follows.

\begin{Theorem}\label{thm:lattice}
For any $\o$-symmetric planar convex body $\MM$ with $3$-fold rotational symmetry and $\lambda\geq 0$, there is a lattice packing $\P$ of translates of $\MM$ such that
\[
\bar{\delta}_2(\mathbf{P}_{\lambda}) = \bar{\delta}_{\M}(\lambda).
\]
In other words, we have $\bar{\delta}_{\M}(\lambda) = \bar{\delta}_{\M}^{\mathrm{lattice}}(\lambda)$ for any $\o$-symmetric planar convex body $\MM$ with $3$-fold rotational symmetry and $\lambda\geq 0$.
\end{Theorem}

\begin{Remark}
As a special case it was proved in \cite{Be-La-2015} that $\bar{\delta}_{\Eu^2}(\lambda) = \bar{\delta}_{\Eu^2}^{\mathrm{lattice}}(\lambda)$ holds for all $\lambda\geq 0$, which is attained for a regular triangle lattice packing of soft disks of radii $1+\lambda$ with soft parameter $\lambda$ in $\Eu^2$.
\end{Remark}

It is natural to expect that Theorem~\ref{thm:lattice} extends as follows.
\begin{Conjecture}\label{open1}
Prove that $\bar{\delta}_{\M}(\lambda) = \bar{\delta}_{\M}^{\mathrm{lattice}}(\lambda)$ holds for any $\o$-symmetric planar convex body $\MM$ and $\lambda\geq 0$.
\end{Conjecture}

The analogue of Problem~\ref{open1} in dimensions $3$ and higher is wide open, even for Euclidean balls. We have the following partial results on soft ball packings in $\Eu^3$.

On the one hand, B\"or\"oczky \cite{Bo} has proved that $\gamma(\BB^3)=\sqrt{\frac{5}{3}}$, which is attained for a BCC lattice packing of $\BB^3$. This implies that  $\bar{\delta}_{\Eu^3}(\lambda) = \bar{\delta}_{\Eu^3}^{\mathrm{lattice}}(\lambda)=1$ holds for all $\sqrt{\frac{5}{3}}-1=0.2909...\leq \lambda$. On the other hand, Theorems 5 and 6 of \cite{Be-La-2015} give upper bounds strictly less than $1$ for $\bar{\delta}_{\Eu^3}(\lambda)$ and for any $0\leq \lambda\leq\sqrt{\frac{3}{2}}-1=0.2247...$. This leaves us with the trivial upper bound $\bar{\delta}_{\Eu^3}(\lambda)\leq 1$ for any $\sqrt{\frac{3}{2}}-1=0.2247...<\lambda<\sqrt{\frac{5}{3}}-1=0.2909...$. We improve this trivial upper bound as follows.

\begin{Theorem}\label{thm:density}
Let $0 < \lambda < \sqrt{\frac{5}{3}}-1=0.2909...$. Then $\bar{\delta}_{\Eu^3}^{\mathrm{lattice}}(\lambda)\leq\bar{\delta}_{\Eu^3}(\lambda) \leq 1 - \left(\frac{ \sqrt{\frac{5}{3}} - 1 - \lambda}{11 \sqrt{\frac{5}{3}} +3 - \lambda}\right)^3$.
\end{Theorem}

Perhaps not surprisingly, the {\it face centered cubic} (FCC) lattice is locally extremal among lattices from the point of view of soft densities as stated in the next theorem.

\begin{Theorem}\label{thm:localmaximum}
For any $\lambda > 0$, among lattice packings of soft balls of radii $1+\lambda$ with soft parameter $\lambda$ in $\Eu^3$, the soft density function has a local maximum at the corresponding FCC lattice.
\end{Theorem}

We cannot resist to raise the problem of finding the largest soft density of lattice packings (resp., packings) of soft balls of radii $1+\lambda$ with soft parameter $\lambda$ for given $\lambda>0$ in $\Eu^3$.

\begin{Conjecture}
Prove the existence of $\lambda_0>0$ such that for any $0<\lambda \leq \lambda_0$, among all lattice packings (resp., packings) of soft balls of radii $1+\lambda$ with soft parameter $\lambda$ in $\Eu^3$, the soft density function has a maximum at the corresponding FCC lattice.
\end{Conjecture}

\begin{Remark}
We note that the method of the proof of Theorem~\ref{thm:localmaximum} presented in Section~\ref{local-max} works under two conditions: every facet of the Dirichlet-Voronoi cell of a ball belongs to a minimal vector of the lattice, and that the symmetry group of the lattice acts transitively on the faces. Thus, Theorem~\ref{thm:localmaximum} holds also for all known densest lattice packings, namely for the $D_4, D_5, E_6, E_7, E_8$ lattices and the Leech lattice.
\end{Remark}

In the rest of this note we prove the theorems stated. We shall use the standard notations ${\rm int}(\cdot), {\rm bd}(\cdot), {\rm cl}(\cdot)$, and ${\rm conv}(\cdot)$ for the interior, boundary, closure, and convex hull of the corresponding sets.

\section{Proof of Theorem~\ref{thm:lattice}}

The proof relies on the properties of Voronoi and Delaunay tesselations in normed planes; for a broader introduction on this topic, the interested reader is referred to \cite{AKL} or \cite{Lambert}. Recall that if $\p$ and $\q$ are points of the normed plane $\M$ with unit disk $\MM$, then $\|\p-\q\|_{\MM}$ stands for the $\MM$-distance from $\p$ to $\q$. Without loss of generality, we may assume that $\MM$ is smooth and strictly convex. Observe that in this case, for any points $\p \neq \q$, any line parallel to $[\p,\q]$ contains a unique point whose $\MM$-distance from $\p$ and $\q$ are equal,
where $[\p,\q]$ stands for the (closed) line segment connecting the points $\p$ and $\q$. Thus, the bisector of any segment is homeomorphic to a line. Furthermore, we assume that $0 < \lambda <  \gamma(\MM)-1$, as otherwise the statement of Theorem~\ref{thm:lattice} is trivial. 

Let $\mathcal{C} \subset \M$ be a point set such that $\mathcal{P}=\{ \c + \MM: \c \in \C \}$ is a packing of translates of $\MM$. We show that there is a lattice packing $\P'$ satisfying $\bar{\delta}_2(\PP_{\lambda}) \leq \bar{\delta}_2(\PP'_{\lambda})$.
Clearly, we may assume that $\P$ is \emph{saturated}, i.e. the $\MM$-distance of any point of $\M$ is strictly less than two from some point of $\C$, as otherwise we may add additional translates of $\MM$ to $\P$.

For any $\c \in \C$, the \emph{Voronoi cell} of $\C$ is defined in the usual way as $V_{\c} = \{ \x \in \M : ||\x-\c||_{\MM} \leq ||\x -\c' ||_{\MM} \hbox{ for any } \c' \in \C \}$.
The Voronoi cell of a point $\c$ is not necessarily convex, but it is starlike with respect to $\c$; i.e. it satisfies the property that for any $\x \in V_{\c}$, $[\c,\x] \subset V_{\c}$ holds. It is also known that in the above case the Voronoi cells of the points of $\C$ form a locally finite decomposition of $\M$. The \emph{Delaunay tessellation} assigned to $\C$ is the `dual' of the Voronoi decomposition; there, two points $\c,\c' \in \C$ are connected by the edge $[\c,\c']$ if and only if the unique smallest homothetic copy of $\MM$ containing $\c,\c'$ contains no other point of $\C$. This also defines a locally finite decomposition of $\M$ into convex polygons, called \emph{Delaunay cells}, which satisfy the following (see also \cite{Lambert}):
\begin{itemize}
\item[(i)] the vertices of any Delaunay cell $F$ are points in $\C$;
\item[(ii)] for any Delaunay cell $F$ there is a homothetic copy $\x + \lambda \MM$, containing all vertices of $F$ in its boundary, which contains no point of $\C$ apart from the vertices of $F$; this homothet is called the \emph{circumdisk} of $F$.
\end{itemize}

In the next part we introduce another locally finite, edge-to-edge decomposition, which we call the \emph{Moln\'ar decomposition}, as follows. This decomposition, in the case that $\MM$ is a Euclidean disk, was introduced in \cite{Molnar}.

Let $F$ be a cell of the Delaunay decomposition, and let us denote the circumdisk of $F$ by $\MM_F \subset \M$, and the center of $\MM_F$ by $\o_F$.
If $F$ does not contain $\o_F$, then there is a unique side of $F$ that separates it from $F$ in $\MM_F$. We call this side a \emph{separating side} of $F$. If $[\c_i, \c_j ]$ is a separating side of $F$, then we call the polygonal curve $[\c_i, \o_F ] \cup [\o_F, \c_j]$ the \emph{bridge} associated to the separating side $[\c_i, \c_j ]$ of $F$ (cf. Figure~\ref{fig:molnar}).

\begin{figure}[ht]
\begin{center}
 \includegraphics[width=0.3\textwidth]{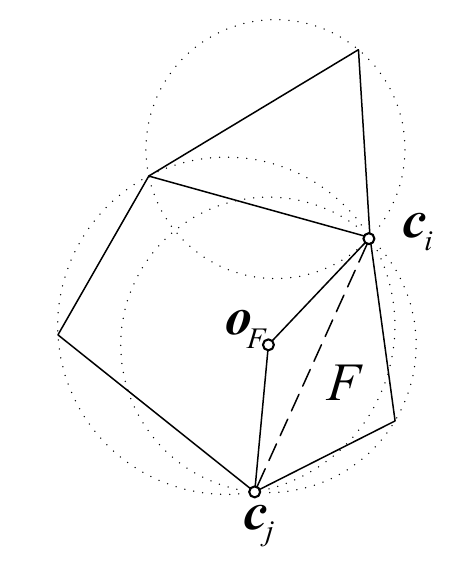}
 \caption{An illustration for the Moln\'ar decomposition of a point system in the case that $\MM$ is a Euclidean disk. In the figure, the edges of the cells, the circumcircles of the cells and the separating sides are denoted by solid, dotted and dashed lines, respectively.}
\label{fig:molnar}
\end{center}
\end{figure}

Our main lemma about the Moln\'ar decomposition is the following.

\begin{Lemma}\label{lem:K_Molnar}
If we replace all separating sides of a Delaunay decomposition by the corresponding bridges, we obtain a locally finite, edge-to-edge decomposition of $\M$.
\end{Lemma}

\begin{proof}
As in \cite{Molnar}, the proof is based on showing the following two statements.
\begin{itemize}
\item[(a)] Bridges may intersect only at their endpoints.
\item[(b)] A bridge may intersect an edge in the Moln\'ar decomposition only at its endpoints.
\end{itemize}

To show (a), recall that for any cell $F$, the points of $\C$ closest to $\o_F$ (measured in normed distance) are exactly the vertices of $F$. Thus, if $[\o_F, \c_i]$ and $[\o_{F'}, \c_j]$ are components of bridges with $\c_i \neq \c_j$ and $\o_F \neq \o_{F'}$, then $\c_i$ is closer to $\o_F$ than to $\o_{F'}$, and $\c_j$ is closer to $\o_{F'}$ than to $\o_F$. Hence, if $Z$ is the set of points in $\M$ closer to $\o_F$ than to $\o_{F'}$, then $Z$ contains $[\o_F,\c_i]$ but it is disjoint from $[\o_{F'}, \c_j ]$, implying that $[\o_F,\c_i]$ and $[\o_{F'}, \c_j ]$ are disjoint.

Now, we show (b), and let $[ \c_i, \c_j]$ be a separating side of the cell $F$. First, observe that by the definition of separating side, the triangle $T$ with vertices $\c_i$, $\c_j$ and $\o_F$ intersects finitely many cells of the Delaunay-decomposition at a point different from $[ \c_i, \c_j]$, each of which is different from $F$. Let $F'$ be the cell adjacent to $F$ and having $[ \c_i, \c_j]$ as a side. Let $\MM_{F'}$ be the circumdisk of $F'$.
By the properties of bisectors and Voronoi cells described at the beginning of the section, the triangle $T = \conv \{ \c_i, \c_j, \o_F\}$ is contained in $\MM_{F'}$, and it is easy to see that it is contained also in the interior of $T'=\conv \{ \c_i, \c_j, \o_{F'} \}$ apart from $[\c_i,\c_j]$. Thus, on the one hand, if $\o_{F'} \in F'$, then $T \setminus [\c_i,\c_j] \subseteq \inter (F')$, and (b) clearly holds for the bridge associated to $[\c_i,\c_j]$. On the other hand, if $\o_{F'} \notin F'$, then there is a separating side $[\c_{i'},\c_{j'}]$ of $F'$, removed during the construction, and this side is the unique side containing any point of the bridge apart from $\c_i$ and $\c_j$. Thus, no side of $F'$ in the Moln\'ar decomposition contains an interior point of the bridge. Furthermore, observe that if any side in the Moln\'ar decomposition contains an interior point of $[\c_i,\o_F] \cup [\o_F, \c_j]$, then it also contains an interior point of $[\c_{i'}, \o_{F'}] \cup [\o_{F'}, \c_{j'}]$. Consequently, we may repeat the argument with $F'$ playing the role of $F$, and since only finitely many cells may intersect $T$ at a point different from $[ \c_i, \c_j]$, and each such cell is different from $F$, we conclude that $[ \c_i, \o_F] \cup [\o_F, \c_j]$ may intersect any side of the Moln\'ar decomposition only at $\c_i$ or $\c_j$.
\end{proof}

Let $[\c_i,\c_j]$ be an edge of a Delaunay cell $F$, and let $T(\c_i,\c_j)$ be an equilateral triangle with edgelength $2$ such that an edge of $T(\c_i,\c_j)$ is parallel to $[\c_i,\c_j]$. Clearly, the circumradius of $T(\c_i,\c_j)$, which we denote by $R(\c_i,\c_j)$, is not greater than the circumradius of $F$. Furthermore, as $\MM$ has $3$-fold rotational symmetry, $T(\c_i,\c_j)$ is a regular triangle in the underlying Euclidean norm as well.

Let us dissect every cell $F$ of the Moln\'ar decomposition by connecting the circumcenter of the corresponding Delaunay cell to each vertex of $F$. We call the so-obtained decomposition the \emph{refined Moln\'ar decomposition} associated to $\C$. Note that then every cell $F$ of this decomposition is of the form $\cl \left(\conv\{\a,\b,\c \} \setminus \conv\{\a,\b,\c' \}\right)$, where $\a,\b \in \C$, $\c$ is the center of a Delaunay cell, and $\c'$ is the center of another Delaunay cell or the midpoint of $[\a,\b]$, furthermore, we have $2 > ||\b-\c||_{\MM}=||\a-\c||_{\MM} \geq R(\a,\b)$, $||\b-\c'||_{\MM}=||\a-\c'||_{\MM}$, and $||\a-\b||_{\MM} \geq 2$. We note that since $0 < \lambda < \gamma(\MM)-1$, $\c$ does not lie in the interior of any soft disk centered at a point of $\C$.
We set
\[
\delta(F) := \frac{\vol_2\left( \left((\a+(1+\lambda)\MM) \cup (\b+(1+\lambda)\MM)\right) \cap F \right)}{\vol_2(F)}.
\]
To prove Theorem~\ref{thm:lattice}, it is sufficient to prove that for any cell $F$, $\delta(F)$ is maximal with the direction of $[\a,\b]$ fixed, if and only if $||\a-\b||_{\MM} = 2$, $\c'$ is the midpoint of $[\a,\b]$, and $||\b-\c||_{\MM}=||\a-\c||_{\MM} = R(\a,\b)$. This statement is an immediate consequence of Lemmas~\ref{lem:legs} and \ref{lem:base}, and it implies that $\bar{\delta}_{\M}(\lambda)= \bar{\delta}_{\M}^{\mathrm{lattice}}(\lambda)$, which is attained by a regular triangle lattice packing of soft disks with soft parameter $\lambda$ in $\M$.

\begin{Lemma}\label{lem:legs}
Let $T=\conv \{ \a,\b,\c \}$ and $T'= \conv \{ \a,\b,\c' \}$ be isosceles triangles (in the norm of $\M$) with base $[\a,\b]$, where $||\a-\b||_{\MM} \geq 2$, and assume that $\c$ is farther from $\a,\b$ than $\c'$. Let $\rho(T) := \frac{\vol_2\left(T \cap \left( (\a+(1+\lambda) \MM) \cup (\b+(1+\lambda) \MM) \right)\right)}{\vol_2(T)}$, and define $\rho(T')$ similarly. Then $\rho(T) \leq \rho(T')$, with equality if and only if $T \subseteq (\a+(1+\lambda) \MM) \cup (\b+(1+\lambda) \MM)$.
\end{Lemma}

\begin{proof}
Without loss of generality, we may assume that $\c' \in T$. Let $||\c'-\b||_{\MM}=||\c'-\a||_{\MM} = \mu$. We consider only the case that $\mu \geq 1+ \lambda$, in the opposite case a similar argument can be given. Note that under this condition the density $\rho(T')$ of $(\a+(1+\lambda) \MM) \cup (\b+(1+\lambda)\MM)$ in $T'$ is greater than $\frac{(1+\lambda)^2}{\mu^2}$, and in the region $\cl( T \setminus T')$ it is less than $\frac{(1+\lambda)^2}{\mu^2}$ (see Figure~\ref{fig:legs}). This implies the assertion.
\end{proof}

\begin{figure}[h]
\centering
\begin{minipage}{0.35\textwidth}
\centering
\includegraphics[width=0.8\textwidth]{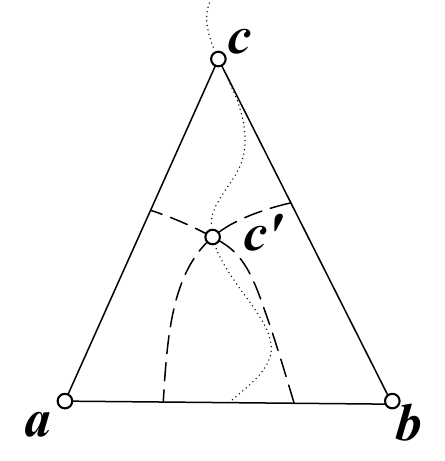}
\caption{An illustration for Lemma~\ref{lem:legs}. The dotted line indicates the bisector of $[\a,\b]$. Dashed lines show the boundaries of $\a + \mu \MM$ and $\b + \mu \MM$.}
\label{fig:legs}
\end{minipage}
\hskip0.08\textwidth
\begin{minipage}{0.45\textwidth}
\centering
\includegraphics[width=0.95\textwidth]{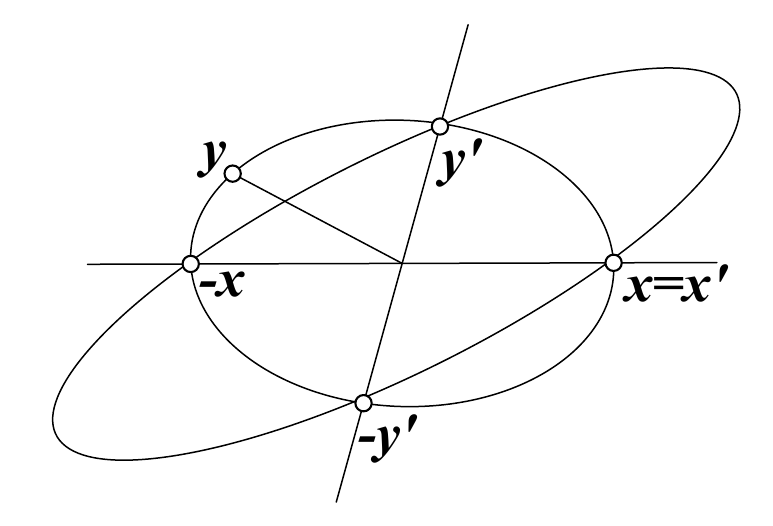}
\caption{An illlustration for Lemma~\ref{rem:geomobs} for the case $\x=\x'$.}
\label{fig:base}
\end{minipage}
\end{figure}

In the following lemma, for any $\p,\q \in \bd \MM$ which are not antipodal, $\widehat{\p \q}$ denotes the shorter arc in $\bd \MM$ connecting $\p$ and $\q$.

\begin{Lemma}\label{rem:geomobs}
Let $\x, \x', \y', \y \in \bd \MM$ in this counterclockwise order in $\bd \MM$ such that no two of them are antipodal, $\widehat{\x'\y'} \subsetneq \widehat{\x\y}$, and $\x' \neq \y'$.
Let $A$ be the linear transformation defined by $A(\x)=\x'$ and $A(\y)=\y'$. Then $A(\widehat{\x\y}) \cap \inter \MM = \emptyset$, and $A(\widehat{\y(-\x)}) \setminus \{ \y',-\x' \}, A(\widehat{(-\y)\x}) \setminus \{ -\y', \x' \} \subset \inter \MM$.
\end{Lemma}

\begin{proof}
The proof is based on two observations. First, note that without loss of generality, we may assume that $\x=\x'$ (see Figure~\ref{fig:base}). Second, we observe that when $\y$ moves to $\y'$ under $A$, every point of $\widehat{\x\y}$ is translated by a positive scalar multiple of the vector $\y'-\y$. This shows that $A(\widehat{\x\y})$ is disjoint from $\inter \MM$. To prove the second statement, we may apply a similar argument.
\end{proof}

\begin{Lemma}\label{lem:base}
Let $T=\conv \{ \a,\b,\c \}$ and $T'= \conv \{ \a',\b',\c \}$ be isosceles triangles (in the norm of $\M$) such that their bases $[\a,\b]$ and $[\a',\b']$ are parallel, $||\c-\b||_{\MM}=||\c-\a||_{\MM} = ||\c-\b'||_{\MM}=||\c-\a'||_{\MM} = \tau > 1+\lambda$, and $||\b-\a||_{\MM} \geq ||\b'-\a'||_{\MM} \geq 2$. Let $\rho(T) := \frac{\vol_2\left(T \cap \left( (a+(1+\lambda) \MM) \cup (b+(1+\lambda) \MM) \right)\right)}{\vol_2(T)}$, and define $\rho(T')$ similarly. Then $\rho(T) \leq \rho(T')$, with equality if and only if $[\a,\b]=[\a',\b']$.
\end{Lemma}

\begin{proof}
Without loss of generality, we may assume that $\c=\o$, and $\a,\a',\b',\b,-\a$ are in this counterclockwise order around $\o$ in $\tau \bd \MM$. Let $A$ be the linear transformation defined by $A(\a)=\a'$, $A(\b) = \b'$. Then $A(T)=T'$, and by Lemma~\ref{rem:geomobs}, we have $A(\a+(1+\lambda) \MM) \cap T' \subseteq (\a'+(1+\lambda)\MM) \cap T'$ and $A(\b+(1+\lambda) \MM) \cap T' \subseteq (\b'+(1+\lambda)\MM) \cap T'$. Since $A$ is nondegenerate, and a nondegenerate linear transformation does not change the ratio of areas of figures, it follows that $\rho(T) \leq \rho(T')$. Furthermore, $\rho(T) = \rho(T')$ implies $A(\a+(1+\lambda) \MM) \cap T' = (\a'+(1+\lambda)\MM) \cap T'$ and $A(\b+(1+\lambda) \MM) \cap T' = (\b'+(1+\lambda)\MM) \cap T'$, which in turn yields that $\a=\a'$ and $\b=\b'$.
\end{proof}
This completes the proof of Theorem~\ref{thm:lattice}.

\section{Proof of Theorem~\ref{thm:density}}

The core method of our proof is the following statement.

\begin{Theorem}\label{thm:largeball}
Let $\P$ be a unit ball packing in $\Eu^3$, and let $0 < \lambda < \sqrt{\frac{5}{3}}-1=0.290994...$. Then any ball of radius $R_{\lambda} := 11 \sqrt{\frac{5}{3}} +3 - \lambda$ contains a ball of radius $\sqrt{\frac{5}{3}} - 1 - \lambda$ that does not overlap $\P_{\lambda}$.
\end{Theorem}

\begin{proof}
Consider a packing $\P$, which, without loss of generality, we assume to be saturated, and let $B:=\c+R_{\lambda}\BB^3$  be the ball of radius $R_{\lambda}$ and center $\c$.
We show that the ball $\bar{B}:=\c+\left(10 \sqrt{\frac{5}{3}}+4\right)\BB^3$ contains a vertex of a Diriclet-Voronoi cell (in short, DV-cell) of $\P$ whose distance from the centers of the balls of $\P$ is at least $\sqrt{\frac{5}{3}}$.

Let $V$ be a DV-cell that contains $\c$, and call a DV-cell $V'$ a $k$-neighbor of $V$ if there is a continuous curve, starting at an interior point of $V$ and ending at an interior point of $V'$ that contains the points of at most $k-1$ other DV-cells of $\P$. In particular, we say that $V$ is a $0$-neighbor of itself.
Notice that the distance of any vertex of a DV-cell from its center is at most $2$, and if it is at most $\sqrt{\frac{5}{3}}$, then the diameter of the cell is at most $2\sqrt{\frac{5}{3}}$. It follows that if $V$ has a $k$-neighbor with $k\leq 5$, which has a vertex at least $\sqrt{\frac{5}{3}}$ far from the centers of the balls of $\P$, then $\bar{B}$ contains a vertex which is at least $\sqrt{\frac{5}{3}}$ far from the unit ball centers.
So, assume that no $k$-neighbor of $V$, with $k \leq 5$, has a vertex at a distance at least $\sqrt{\frac{5}{3}}$ from the centers of the balls in $\P$. We are going to derive a contradiction from it, which then completes the proof of Theorem~\ref{thm:largeball}. In what follows $k \leq 5$.

Consider a vertex $\p$ of a $k$-neighbor of $V$. Then, by the same argument as in \cite{Bo} (see the second paragraph on page 87), $\p$ belongs to at most four DV-cells of $\P$. The third paragraph on page 87 of  \cite{Bo} implies that the faces of a $k$-neighbor of $V$ are at most hexagons.

We show that there is a $k$-neighbor of $V$ which has a hexagon face.
Similarly like in \cite{Bo} (see the fourth paragraph on page 87), by Euler's formula if a DV-cell has more than 12 faces, then it has a face with more than $5$ vertices, which must be, by the considerations above, a hexagon.
Using the argument in \cite{Bo} (see the fifth paragraph on page 87), if a $k$-neighbor $\bar{V}$ of $V$ has less than $11$ faces, then it has a vertex at a distance greater than $\sqrt{\frac{5}{3}}$ from the center of $\bar{V}$, which contradicts our assumption.
Using the sixth paragraph on page 87 in \cite{Bo}, the possibility that $\bar{V}$ has $11$ faces is exluded by Lemma 2 of \cite{Bo}.
Based on the last paragraph of page 87 in  \cite{Bo}, consider the case that each $k$-neighbor of $V$ has exactly $12$ faces, and each face is a pentagon.
Using the idea of Lemma 3 of \cite{Bo}, starting with $V$ we can construct a cell-complex $\mathcal{P}$ isomorphic to that of the regular 120-cell in $\Eu^4$.
Note that then each cell in $\mathcal{P}$ is a $k$-neighbor of $V$, with $k \leq 5$. Thus, at each vertex in this complex, exactly four cells of $\P$ meet, which, similarly like in the proof of Lemma 3 of \cite{Bo}, yields a contradiction.

Thus, we can conclude that a $k$-neighbor of $V$ has  a hexagon face. Then similarly like in \cite{Bo} (see page 88), a vertex of this hexagon has distance at least $\sqrt{\frac{5}{3}}$ from the center of the cell, which contradicts our assumption. This completes the proof of Theorem~\ref{thm:largeball}.
\end{proof}

Now, let $\P$ be a unit ball packing in $\Eu^3$, and let $0 < \lambda < \sqrt{\frac{5}{3}}-1$. We are going to show that $\bar{\delta}_3 (\P_\lambda) \leq 1 - \frac{\left( \sqrt{\frac{5}{3}} - 1 - \lambda\right)^3}{R_{\lambda}^3}$. The details are as follows.

Let $\delta_{\lambda} := 1 - \frac{\left( \sqrt{\frac{5}{3}} - 1 - \lambda\right)^3}{R_{\lambda}^3}$.
By Theorem~\ref{thm:largeball}, it follows that any ball of radius $R_{\lambda}$ contains a ball of radius $\sqrt{\frac{5}{3}} - 1 - \lambda$ that does not overlap $\P_{\lambda}$. In other words, we have that
\[
\frac{\int_{\z + R_{\lambda} \BB^3} \chi_{\P_{\lambda}}(\y) \, d \y}{\frac{4}{3}R_{\lambda}^3 \pi} \leq \delta_{\lambda}
\]
for any $\z \in \Eu^3$, where $\chi_{\P_{\lambda}}$ is the indicator function of $\P_{\lambda}$.

Let $R>0$ be sufficiently large. Then the density of $\P_{\lambda}$ in $R \BB^3$ is
\[
\frac{\int_{R \BB^3} \chi_{\P_{\lambda}} (\y) \, d \y}{\frac{4}{3} R^3 \pi} =
\frac{\int_{R \BB^3} \int_{\y + R_{\lambda} \BB^3} \chi_{\P_{\lambda}}(\y) \, d \z \, d \y}{\frac{16}{9} R^3 R_{\lambda}^3 \pi^2} =
\frac{\int_{(R+R_{\lambda}) \BB^3}  \int_{(\z + R_{\lambda} \BB^3) \cap R \BB^3} \chi_{\P_{\lambda}}(\y) \, d \y \, d \z}{\frac{16}{9} R^3 R_{\lambda}^3 \pi^2} \le\]
\[
\leq \frac{\int_{(R+R_{\lambda}) \BB^3} \delta_{\lambda} \, d \z}{\frac{4}{3} R^3 \pi} =
\delta_{\lambda} \cdot \left( 1 + \frac{R_{\lambda}}{R} \right)^3.
\]
This completes the proof of Theorem~\ref{thm:density}.

\section{Proof of Theorem~\ref{thm:localmaximum}}\label{local-max}

In the proof we use the following formula of Csik\'os, stated as Theorem 4.1 in \cite{Csikos}.

\begin{Theorem}\label{thm:csikos}
Let $\{ \mathbf{B}_i^d=\mathbf{x}_i + r_i \mathbf{B}^d : i=1,2,\ldots, n\}$ be a family of balls in $\Eu^d$. For any value $i$ and point $\mathbf{p} \in \Eu^d$, set $K_i(\mathbf{p}) = \| \mathbf{x}_i - \mathbf{p}  \|^2 - r_i^2$. For any $i$, let $C_{i} = \{ \mathbf{p} \in \Eu^d : K_i(\mathbf{p}) \leq K_j(\mathbf{p}), j=1,2,\ldots,n \}$ be the \emph{weighted Dirichlet-Voronoi} cell of $\mathbf{B}_i^d$. If $i \neq j$, the set $W_{ij} = C_i \cap C_j \cap \mathbf{B}_i^d \cap \mathbf{B}_j^d$ is called the \emph{wall} between $\mathbf{B}_i^d$ and $\mathbf{B}_j^d$. Let $\gamma_i : (a,b) \to \Eu^d$, $i=1,2,\ldots, n$ be smooth curves, and $t_0 \in (a,b)$ be a value such that the points $\gamma_i(t_0) = \mathbf{x}_i$ are pairwise distinct. Then the function
$v(t) :=\vol_d \left( \bigcup_{i=1}^n \left( \gamma_i(t) + r_i \mathbf{B}^d \right) \right)$ is differentiable at $t_0$, and its derivative is
\[
v'(t_0) = \sum_{1 \leq i < j \leq n} \|\gamma_i(t_0) - \gamma_j(t_0) \|' \cdot \vol_{d-1}\left( W_{ij} \right).
\]
\end{Theorem}

We note that Theorem~\ref{thm:csikos} remains valid for one-sided derivatives at $a$ and $b$ if the curves $\gamma_i$ are defined in the interval $[a,b]$.

Let $\mathcal{F}$ be an FCC lattice packing of soft balls containing the soft ball $(1+\lambda)\mathbf{B}^3$ such that $\BB^3$ is touched by the unit balls $\{ \BB_i^3 : i=1,2,\ldots,12\}$ concentric to $12$ soft balls of $\mathcal{F}$. Moreover, let $V$ be the Dirichlet-Voronoi cell of $\mathbf{B}^3$ in $\mathcal{F}$. 

Without loss of generality, we may assume that $\lambda$ is chosen such that the soft balls of $\mathcal{F}$ do not cover $\Eu^3$.

Consider an arbitrary lattice packing $\mathcal{F}'$ of soft balls of radii $1+\lambda$ that is `close' to $\mathcal{F}$ sharing $(1+\lambda)\BB^3$ with $\mathcal{F}$. Then there is a smooth deformation $\mathcal{F}(t)$, $t \in [0,1]$ such that $\mathcal{F}(0)=\mathcal{F}$, $\mathcal{F}(1)=\mathcal{F}'$, and for every $t$
the length of the minimal lattice vector in $\mathcal{F}(t)$ is at least $2$.
Then the centers of the neighbors of $\BB^3$ move on smooth curves; that is, there are smooth curves $\gamma_i (t): [0,1] \to \Eu^3$, $i=1,2,\ldots, 12$ such that $\gamma_i(0)$ is the center of $\BB_i^3$. Recall that any lattice packing of unit balls with each ball touched by $12$ others is an FCC lattice. (See for example, \cite{BoSz}.) Thus, without loss of generality one may assume that $\mathcal{F}(t)$ is not congruent to $\mathcal{F}$ for any $0<t \leq 1$ and therefore for some value of $i$, $d'_i:=\| \gamma_i(t) \|'_{t=0} > 0$.

Let the Dirichlet-Voronoi cell of $\mathbf{B}^3$ in $\mathcal{F}(t)$ be denoted by $V(t)$.
The soft density of $\mathcal{F}(t)$ is equal to $\rho(\mathcal{F}(t)) = \frac{\vol_3 \left( V(t) \cap (1+\lambda) \mathbf{B}^3 \right)}{\vol_3 \left( V(t) \right)}$.
Let $F_i$ and $W_i$ denote, respectively, the face of $V$ touched by $\mathbf{B}_i^3$, and $((1+\lambda) \mathbf{B}^3) \cap F_i$.
Note that since the symmetry group of the FCC lattice acts transitively on the faces of $V$, there are quantities $F$, $W$ such that $F=\area(F_i)$ and $W=\area(W_i)$ independently of the values of $i$.

As $ \|\gamma_i(0) \| = 2$ for every $i$, we have $\vol_{3}(V)= \frac{1}{3} \sum_{i=1}^{12} \frac{1}{2} \|\gamma_i(0) \| \area(F_i)) = 4F$, and
$\vol_3 \left( V \cap (1+\lambda) \mathbf{B}^3 \right) > 4W$.
Applying Theorem~\ref{thm:csikos}, we obtain that
\[
\left( \vol_3 \left( V(t) \cap (1+\lambda) \mathbf{B}^3 \right) \right)'_{t=0} = \frac{1}{2} \sum_{i=1}^{12} d'_i W.
\]
Similarly, we have
\[
\left( \vol_3 \left( V(t) \right) \right)'_{t=0} = \frac{1}{2} \sum_{i=1}^{12} d'_i  F.
\]
Thus, as $d'_i \geq 0$ for all values of $i$, and $d'_i > 0$ for some values of $i$, it follows that
\[
\rho(\mathcal{F}(t))'_{t=0} < \frac{2F \sum_{i=1}^{12} d'_i W - 2W \sum_{i=1}^{12} d'_i F}{16 F^2} = 0.
\]
This completes the proof of Theorem~\ref{thm:localmaximum}.

\bigskip

 {\large {\bf Acknowledgements}}

The authors would like to thank the anonymous referee for careful reading and valuable comments.

\bigskip

\noindent K\'aroly Bezdek \\
\small{Department of Mathematics and Statistics, University of Calgary, Canada}\\
\small{Department of Mathematics, University of Pannonia, Veszpr\'em, Hungary\\
\small{E-mail: \texttt{kbezdek@ucalgary.ca}}

\bigskip

\noindent and

\bigskip

\noindent Zsolt L\'angi \\
\small{MTA-BME Morphodynamics Research Group}\\ 
\small{Department of Algebra and Geometry}\\ 
\small{Budapest University of Technology and Economics, Budapest, Hungary}\\
\small{\texttt{zlangi@math.bme.hu}}

\end{document}